\documentclass[11pt,reqno]{amsart}
\usepackage[margin=1.25in, footskip=.5in]{geometry}
\usepackage[T1]{fontenc}
\usepackage[american]{babel}
\usepackage{amsmath,amsfonts,amssymb,amsthm,bm,wasysym}
\usepackage{mathrsfs,bbm,graphicx}
\usepackage{enumitem}
\usepackage{multirow}
\usepackage{caption,float}
\usepackage{subcaption}
\usepackage{tikz}
\usepackage[colorlinks,linkcolor=cyan!60!black,citecolor=cyan!50!black,pagebackref,hypertexnames=false, breaklinks]{hyperref}

\allowdisplaybreaks
\newtheorem{theorem}{Theorem}
\newtheorem{proposition}[theorem]{Proposition}

\newtheorem{lemma}[theorem]{Lemma}
\theoremstyle{definition}

\newtheorem*{def*}{Definition}

\newtheorem{remark}{Remark}

\newtheorem*{theorem*}{Theorem}
\newtheorem*{lemma*}{Lemma}
\newtheorem*{claim*}{Claim}
\newtheorem*{prop*}{Proposition}
\newtheorem*{rem*}{Remark}

\numberwithin{equation}{section}

\title{Dyadic fractional Sobolev spaces:\\ Embeddings and algebra property}
\author{Patricia Alonso Ruiz and Valentia Fragkiadaki}

\date{\today}

\makeindex
\begin{document}

\begin{abstract}
    This paper studies a dyadic version of the classical fractional Sobolev spaces in $\mathbb{R}^n$ for $n\geq 1$. It provides new proofs of the corresponding fractional Sobolev embedding as well as the algebra property of the spaces, which rely solely on dyadic techniques and bypass the Fourier transform. Counterexamples are constructed to verify the failure of the algebra property in low-regularity ranges.  
\end{abstract}
\maketitle

\tableofcontents


\section{Introduction}
Sobolev spaces may be regarded as a point of intersection between the two broad fields of harmonic analysis and partial differential equations (PDEs). On the one hand, Sobolev spaces are fundamental to the theory of PDEs as they are customary spaces of solutions to certain equations. On the other hand, many tools used to study these spaces, such as the Fourier transform or the Littlewood-Paley theory, come from the field of harmonic analysis. In some instances, it may happen that a tool like the Fourier transform is not available or straightforward, while, in turn, another tool like a system of Haar functions is more accessible. With that situation in mind, this paper focuses in using dyadic harmonic analysis to study (dyadic) fractional Sobolev spaces. By developing purely dyadic proofs of corresponding known results from the classical case, one bypasses definitions via Fourier transform and offers an alternative approach to extend definitions and results to settings without a well-defined or straightforward concept of Fourier transform.

\medskip

Dyadic fractional Sobolev spaces were first introduced to study dyadic nonlocal Schr\"odinger equations in a series of works~\cite{ABG13,AA15,ACGN23}. There, partial derivatives were defined in terms of Haar multipliers and dyadic homogeneous integral operators, replacing the commonly used frequency space involving the Fourier transform by a dyadic analogue involving the Haar system. Starting off with that dyadic definition of fractional Sobolev space in $\mathbb{R}$, it was possible to give in~\cite{ARF24} a new proof of the Sobolev embedding theorem and the algebra property of the dyadic Sobolev space $H_{\mathcal{D}}^s(\mathbb{R})$ that only involved dyadic techniques. However, the results in \cite{ARF24} were restricted to the one dimensional Euclidean space, limiting their potential use in applications, such as finding solutions of PDEs in higher dimensions. It is already well-known that in the continuous case, the respective results hold true for the general $n-$dimensions. This paper extends our previous dyadic results from~\cite{ARF24} to the $n$-dimensional case. Dyadic characterizations of different function spaces involving smoothness such as Triebel-Lizorkin and Besov spaces have been investigated by Garrig\'os, Seeger and Ulrich in~\cite{GSU17,GSU18,GSU23}.

\medskip

The first main result, proved in Section~\ref{S:embeddings}, provides a dyadic proof of the classical fractional Sobolev embedding.

\begin{theorem*}
Let $s>0$ and $1<q<\infty$. Then, it holds that $H^{s}_\mathcal{D}(\mathbb{R}^n)\subseteq L^q(\mathbb{R}^n)$ when
\begin{enumerate}[wide=0em,itemsep=.5em, label={\rm (\roman*)}]
\item $0<s<\frac{n}{2}$ and $q=\frac{2n}{n-2s}$,
\item $s=\frac{n}{2}$ and $1<q<\infty$,
\item $s>\frac{n}{2}$ and $q=\infty$.
\end{enumerate}
\end{theorem*}
Looking at a standard textbook on fractional Sobolev spaces, e.g.~\cite[Section 7]{Leo23}, we see that the result is in complete analogy with the classical one, and it is a generalization of the dyadic result in $\mathbb{R}$ from~\cite{ARF24}.

\medskip

The second  main result, shown in Section~\ref{S:algebra_prop}, is a classification of the spaces $H^s_\mathcal{D}(\mathbb{R}^n)$ in terms of the presence (or absence) of the algebra property. Specifically, this refers to the question whether the space is closed under pointwise multiplication, which is of importance in the study of non-linear PDEs with a power non-linearity, see e.g.~\cite{Caz03} and references therein. 

\begin{theorem*}
    Let $0<s<\frac{n}{2}$. The space $H^s_\mathcal{D}(\mathbb{R}^n)$ is an algebra if and only if $s>\frac{n}{2}$.
\end{theorem*}
Again, the result is in complete analogy with the classical one proved by Strichartz in~\cite{Str67}, and it generalizes the dyadic result for $n=1$ from~\cite{ARF24}.

\medskip

The techniques used in the present paper to extend the results from~\cite{ARF24} have a similar flavor to those developed there and involve in many cases direct computations. Nevertheless, we want to stress that they are not a straightforward adaptation of the latter: In the one-dimensional case, each dyadic interval gives rise to one cancellative Haar function, whereas in the $n$-dimensional counterpart, each dyadic cube produces $2^n-1$ cancellative Haar functions. 
The main difficulty can for example be seen in Lemma \ref{L:Haar_coeff_sq}, when computing the Haar expansion of $f^2$ in order to prove the algebra property results. In the counterexamples showing that $H^s_\mathcal{D}$ is not an algebra for $s\leq n/2$, the way to overcome this new difficulty coming from the complicated formula for $f^2$ is to build our function using only one out of the $2^n-1$ possible Haar functions for each of the selected cubes. However, this is not an option anymore when proving that $H^s_\mathcal{D}$ is an algebra for $s> n/2$, where we have to directly deal with the formula for $f^2$.

\medskip

The paper is organized as follows: Section~\ref{S:defs_back} provides the necessary background for the dyadic setup we use in $\mathbb{R}^n$, and introduces the dyadic fractional Sobolev space $H^s_{\mathcal{D}}(\mathbb{R}^n)$. This section also includes useful equalities needed later on, which may be of independent interest. Section~\ref{S:embeddings} presents the dyadic proofs of each regularity range of the fractional Sobolev embedding and with a similar structure, Section~\ref{S:algebra_prop} analyzes the algebra property of $H^s_{\mathcal{D}}(\mathbb{R}^n)$, constructing explicit counterexamples in the cases where the property fails. To the best of our knowledge, these are genuinely new.

\section{Definitions and background}\label{S:defs_back}
We start by setting up notation, definitions and by providing some basic results in the dyadic setting that will be used throughout the paper. 
\subsection{The dyadic setting}
We consider $\mathbb{R}^n$ equipped with the $n$-dimensional Euclidean measure, denoted by $|\cdot|$, and with the standard dyadic grid
\begin{multline*}
    \mathcal{D}:=\{Q=I_1\times\ldots I_n\colon~I_1,\ldots, I_n\text{ dyadic intervals}\\
    \text{with }|I_1|=\ldots=|I_n|=2^k\text{ for some }k\in\mathbb{Z}\}.
\end{multline*}
The side length of any $Q\in\mathcal{D}$ is denoted by $\ell(Q)$. Given a fixed cube $Q\in\mathcal{D}$ and $k\geq 0$, the dyadic grid of level $k$ adapted to $Q$ is defined as
\begin{equation}\label{E:def_grid_level_k}
    \mathcal{D}_k(Q):=\Big\{P\in \mathcal{D}\colon P\subset Q\text{ and }\ell(P)=2^{-k}\ell(Q)\Big\}.
\end{equation}
In particular, $\frac{|P|}{|Q|}=2^{-kn}$ for any $P\in\mathcal{D}_k(Q)$. We usually write $\mathcal{D}(Q):=\mathcal{D}_0(Q)$ and note that
\begin{equation*}
    \mathcal{D}(Q)=\bigcup_{k=0}^\infty\mathcal{D}_k(Q).
\end{equation*}
The Haar basis of $L^2(\mathbb{R}^n,dx)$ associated with $\mathcal{D}$ is constructed as follows: When $n=1$, for any dyadic interval $I$ with right half $I_+$ and left half $I_-$ we define the functions
\begin{equation*}
    h_I^{0}(x):=\frac{1}{\sqrt{|I|}}\big(\mathbf{1}_{I_{+}}(x)-\mathbf{1}_{I_{-}}(x)\big)
    \quad\text{and}\quad
    h_I^{1}(x):=\frac{1}{\sqrt{|I|}}\mathbf{1}_{I}(x).
\end{equation*}
For a generic $n\geq 1$, $Q=I_1\times\ldots \times I_n\in\mathcal{D}$ and $\varepsilon\in\{0,1\}^n{\setminus}\{(1,\ldots,1)\}$, one now defines
\begin{equation*}
    h_Q^\varepsilon(x):=\prod_{i=1}^nh_{I_i}^{\varepsilon_i}(x_i),\qquad x=(x_1,\ldots,x_n)\in\mathbb{R}^n.
\end{equation*}
Note that if two cubes $P,Q\in\mathcal{D}$ satisfy $P\subsetneq Q$, then $h_Q^{\varepsilon}$ is constant on $P$. We will often denote that constant by $h_Q^\varepsilon(P)$, which equals either $|Q|^{-1/2}$ or $-|Q|^{-1/2}$, whichever $h_Q^\varepsilon(x)$ takes at $x\in P\subsetneq Q$. In addition, for any $\varepsilon\not\equiv 1$,
\begin{equation*}
    \int_Qh_Q^\varepsilon dx=0.
\end{equation*}

The role of Schwarz functions is played by the linear span associated with the basis $\{h_Q^\varepsilon\}_{Q,\varepsilon}$, namely
\begin{equation*}
    \mathscr{S}_{\mathcal{D}}(\mathbb{R}^n):={\rm span}\big\{h_Q^\varepsilon\colon Q\in\mathcal{D},\,\varepsilon\in\{0,1\}^n{\setminus}\{(1,\ldots,1)\}\big\},
\end{equation*}
which is dense in $L^2(\mathbb{R}^n)$. For simplicity, the results presented in the paper will be stated and proved for functions in $\mathscr{S}_{\mathcal{D}}(\mathbb{R}^n)$. 

\medskip

The inner product of two functions $f,g\in \mathscr{S}_{\mathcal{D}}(\mathbb{R}^n)$ and the average over a cube $Q\in\mathcal{D}$ by
\begin{equation*}
    (f,g):=\int_{\mathbb{R}^n}fg\,dx\qquad\text{and}\qquad \langle f\rangle_Q:= \frac{1}{|Q|} \int_Q f \,dx .
\end{equation*}
In this way, any function $f\in\mathscr{S}(\mathbb{R}^n)$ can be expressed by means of its Haar expansion as
\begin{equation*}
    f=\sum_{\substack{Q\in\mathcal{D}\\ \varepsilon\not\equiv 1}} (f, h_Q^\varepsilon) h_Q^\varepsilon
\end{equation*}
and its average by
\begin{equation}\label{E:avg_char}
    \langle f\rangle_Q = \sum_{\substack{P \supsetneq Q\\ \varepsilon\not\equiv 1}} (f, h_P^\varepsilon) h_P^\varepsilon(Q).
\end{equation}
In addition, Haar functions satisfy the relation
\begin{equation}\label{E:HaarP_char}
    h_Q^\varepsilon(x)=\sum_{R\in\mathcal{D}_k(Q)}h_Q^\varepsilon(R)\mathbf{1}_R(x)
\end{equation}
for any $k\geq 1$, which shall be useful later on in proofs.

\begin{remark}
    The condition $\varepsilon\not\equiv 1$ will be dropped from summations over dyadic cubes to ease the notation as long as no confusion may occur.    
\end{remark}
\subsection{Useful equalities}
The next observations, which may be interesting on their own, will be applied in subsequent proofs. They generalize those in~\cite{ARF24} and reveal the exact dimension-(in)dependence of the constants involved.

\begin{proposition}\label{P:weighted_indicator}
Let $n\geq 1$ and $s>0$. For any $Q\in\mathcal{D}$ it holds that
\begin{equation}\label{E:weighted_indicator}
    \sum_{P\subsetneq Q}|P|^{\frac{s}{n}}\mathbf{1}_P(x)=\frac{1}{1-2^{-s}}|Q|^{\frac{s}{n}}\mathbf{1}_Q(x)
\end{equation}
for any $x\in\mathbb{R}^n$.
\end{proposition}
\begin{proof}
    If $x\notin Q$, then clearly both sides of~\eqref{E:weighted_indicator} are zero. Let us thus assume that $x\in Q$. Rewriting the left-hand side with the dyadic grids from~\eqref{E:def_grid_level_k} one gets
    \begin{align*}
     \sum_{\substack{P\subseteq Q}}|P|^s\mathbf{1}_P(x)&=\sum_{k=0}^\infty \sum_{\substack{P\in\mathcal{D}_k(Q)}}|P|^{\frac{s}{n}}\mathbf{1}_P(x)\\
     &=\sum_{k=0}^\infty\sum_{\substack{P\in\mathcal{D}_k(Q)}}2^{-ks}|Q|^{\frac{s}{n}}\mathbf{1}_P(x)\\
     &=|Q|^{s/n}\sum_{k=0}^\infty2^{-ks}\bigg(\sum_{\substack{P\in\mathcal{D}_k(Q)\\}}\mathbf{1}_P(x)\bigg)
     =|Q|^{\frac{s}{n}}\frac{1}{1-2^{-s}}\mathbf{1}_Q(x).
    \end{align*}
    In the last equality we use the fact that the cubes in $\mathcal{D}_k(Q)$ are disjoint.
\end{proof}

\begin{proposition}\label{P:weighted_hI}
Let $n\geq 1$ and $s>0$. For any $Q\in\mathcal{D}$ and $\varepsilon\in\{0,1\}{\setminus}\{(1,\ldots,1)\}$ it holds that
 \begin{equation*}
     \sum_{P\subsetneq Q}|P|^{\frac{s}{n}}h^\varepsilon_Q(P)\mathbf{1}_P(x)=\frac{1}{2^s-1}|Q|^{\frac{s}{n}}h_Q^\varepsilon(x)
 \end{equation*}
 for every $x\in \mathbb{R}^n$.
\end{proposition}
\begin{proof}
    Note first that $h_Q(P)=h_Q(R)$ for any $P\subseteq R\subsetneq Q$. By virtue of Proposition~\ref{P:weighted_indicator},
    \begin{align*}
      \sum_{P\subsetneq Q}|P|^{\frac{s}{n}}h^\varepsilon_Q(P)\mathbf{1}_P(x)&
      =\sum_{R\in \mathcal{D}_1(Q)}\sum_{P\subseteq R}|P|^{\frac{s}{n}}h^\varepsilon_Q(P)\mathbf{1}_P(x)\\
      &=\sum_{R\in \mathcal{D}_1(Q)}\sum_{P\subseteq R}|P|^{\frac{s}{n}}h^\varepsilon_Q(R)\mathbf{1}_P(x)\\
      &=\sum_{R\in \mathcal{D}_1(Q)}h^\varepsilon_Q(R)\sum_{P\subseteq R}|P|^{\frac{s}{n}}\mathbf{1}_P(x)\\
      &=\frac{1}{1-2^{-s}}\sum_{R\in \mathcal{D}_1(Q)}h^\varepsilon_Q(R)|R|^{\frac{s}{n}}\mathbf{1}_{R}(x).
    \end{align*}
    Further, since $|R|=2^{-n}|Q|$, it follows from the above and~\eqref{E:HaarP_char} with $k=1$ that
    \[
    \sum_{P\subsetneq Q}|P|^{\frac{s}{n}}h^\varepsilon_Q(P)\mathbf{1}_P(x)=\frac{2^{-s}}{1-2^{-s}}|Q|^{\frac{s}{n}}\!\!\!\sum_{R\in \mathcal{D}_1(Q)}h^\varepsilon_Q(R)\mathbf{1}_{R}(x)=\frac{1}{2^s-1}|Q|^{\frac{s}{n}}h_Q^\varepsilon(x)
    \]
    as we wanted to prove.
\end{proof}

We finish this section by recording a useful observation. Recall the convention
\begin{equation}\label{E:def_index_sum}
    (\varepsilon+\tilde{\varepsilon})_i:=
    \begin{cases}
        0&\text{if }\varepsilon_i=\tilde{\varepsilon}_i,\\
        1&\text{if }\varepsilon_i\neq \tilde{\varepsilon}_i,\\
    \end{cases}
    \qquad i=1,\ldots,n,
\end{equation}
for any $\varepsilon,\tilde{\varepsilon}\in\{0,1\}^n$.
\begin{lemma}\label{L:product_Haar}
    For any $Q,\tilde{Q}\in\mathcal{D}$ and any $\varepsilon,\tilde{\varepsilon}\in\{0,1\}^n{\setminus}\{(1,\ldots,1)\}$,
    \begin{equation*}
        h_Q^\varepsilon h_Q^{\tilde{\varepsilon}}=h_Q^{\varepsilon+\tilde{\varepsilon}}        
    \end{equation*}
    and
    \begin{equation*}
        h_{Q}^\varepsilon h_{\tilde{Q}}^{\tilde{\varepsilon}}=h_Q^\varepsilon(\tilde{Q})h_{\tilde{Q}}^{\tilde{\varepsilon}},
    \end{equation*}
    for $\varepsilon\not\equiv\tilde{\varepsilon}$ and $\tilde{Q}\subsetneq Q$.
\end{lemma}

\subsection{Fractional Sobolev spaces}
We work with the generalization to an arbitrary dimension $n$ of the dyadic Sobolev spaces considered in~\cite{ABG13,AABG16,ACGN23,ARF24}. We remind the reader that this definition is inspired by the classical (Euclidean) connection in harmonic analysis between the fractional Laplace operator $(-\Delta)^{s/2}$ and the singular operator
\begin{equation*}
I_sf(x):=\int_{\mathbb{R}^n}\frac{1}{|x-y|^{n-s}}f(y)dy,
\end{equation*}
where $0<s<n$. Its dyadic analogue, defined as
\begin{equation*}
T_{\mathcal{D},s}f:=  \sum_{Q \in \mathcal{D}} |Q|^{\frac{s}{n}} \langle f \rangle_Q \mathbf{1}_Q
\end{equation*}
for any $f\in\mathscr{S}(\mathbb{R}^n)$, has been object of study in many works, see e.g.~\cite{MW74,Saw88,HRS16} and references therein. For a precise comparison between $I_sf$ and $T_{\mathcal{D},s}f$ we refer to~\cite[Proposition 3.6]{Cru17}.

\medskip

With the latter connection in mind, we define the dyadic analogue to the fractional derivative $(-\Delta)^{s/2}$ by
\begin{equation*}
    D^s_{\mathcal{D}}f:=\sum_{\substack{Q\in\mathcal{D}\\ \varepsilon\not\equiv 1}}|Q|^{-\frac{s}{n}}(f,h^\varepsilon_Q)h^\varepsilon_Q(x)
\end{equation*}
for $f\in\mathscr{S}(\mathbb{R}^n)$. 
Along the lines of the classical definition of the fractional Sobolev space $H^s(\mathbb{R}^n)$ via Fourier transform, and in analogy to the one-dimensional case~\cite[Theorem 4.1]{AA15}, we now define the seminorm
\begin{equation}\label{E:def_Hs_seminorm}
    \|f\|_{\dot{H}^s_{\mathcal{D}}(\mathbb{R})}:=\|D^s_{\mathcal{D}}f\|_{L^2(\mathbb{R}^n,dx)}= \bigg( \sum_{\substack{Q\in\mathcal{D}\\ \varepsilon\not\equiv 1}} |Q|^{-\frac{2s}{n}}|(f,h^\varepsilon_Q)|^2 \bigg)^{1/2} 
\end{equation}
and the dyadic Sobolev space
\begin{equation}
    H_{\mathcal{D}}^s(\mathbb{R}^n):=\{f\in L^2(\mathbb{R}^n,dx)\colon~D^s_{\mathcal{D}}f\in L^2(\mathbb{R}^n,dx)\}
\end{equation}
equipped with the norm 
\begin{equation}\label{E:def_Hs_norm}
    \|f\|_{H^s_{\mathcal{D}}(\mathbb{R}^n)}:=\big(\|f\|_{\dot{H}^s_{\mathcal{D}}(\mathbb{R}^n)}^2+\|f\|_{L^2(\mathbb{R}^n,dx)}^2\big)^{1/2}.
\end{equation}
As a first observation, we show how that the classical connection
\begin{equation*}\label{E:Muscalu-Schlag}
    f=C_{n,s}I_s(-\Delta)^{s/2}f
\end{equation*}
for some (explicit) constant $C_{n,s}>0$ and any Schwarz function $f$, see e.g.~\cite[p.177]{MS13}, also holds in the dyadic setting. Note, however, that in this case the constant is \emph{independent} of the dimension $n$.
\begin{lemma}\label{L:dyadic_MS}
    Let $0<s<n$. For any $f\in \mathscr{S}_{\mathcal{D}}^s(\mathbb{R}^n)$,
    \[
    f=(2^s-1)T_{\mathcal{D},s}D_{\mathcal{D}}^sf.
    \]
\end{lemma}
\begin{proof}
Consider first $g\in\mathscr{S}_{\mathcal{D}}^s(\mathbb{R}^n)$. Using the Haar expansion of the average from~\eqref{E:avg_char}, interchanging the order of summation and applying Proposition~\ref{P:weighted_hI},
    \begin{align*}
        (2^s-1)T_{\mathcal{D},s}g&=(2^s-1)\sum_{Q \in \mathcal{D}} |Q|^{\frac{s}{n}} \langle g \rangle_Q \mathbf{1}_Q\\
        &=(2^s-1)\sum_{Q \in \mathcal{D}} |Q|^{\frac{s}{n}}\mathbf{1}_Q\sum_{\substack{P \supsetneq Q\\ \varepsilon\not\equiv 1}} (g, h^\varepsilon_P) h^\varepsilon_P(Q)\\
        &=(2^s-1)\sum_{\substack{P \in \mathcal{D}\\ \varepsilon\not\equiv 1}}(g, h^\varepsilon_P)\sum_{Q
        \subsetneq P}|Q|^{\frac{s}{n}}h_P^\varepsilon(Q)\mathbf{1}_Q\\
        &=\sum_{\substack{P \in \mathcal{D}\\ \varepsilon\not\equiv 1}}|P|^{\frac{s}{n}}(g,h_P^\varepsilon)h_P^\varepsilon.
    \end{align*}
    Plugging $g=D^s_{\mathcal{D}}f$ the result follows.
\end{proof}
Second, we note that only cubes of ``small'' volume matter to the $H^s$-norm~\eqref{E:def_Hs_seminorm}.
\begin{lemma}\label{L:norm_equivalence}
    Let $0<s<n$. For any $f\in \mathscr{S}_{\mathcal{D}}(\mathbb{R}^n)$ 
    \begin{equation}\label{E:seminorm_equivalence}
    \frac{1}{2}\|f\|_{H^s_{\mathcal{D}}(\mathbb{R}^n)}^2\leq \|f\|_{L^2}^2 + \sum_{\substack{|Q|<1 \\ \varepsilon\not\equiv 1}} |Q|^{-\frac{2s}{n}}|(f,h_Q^\varepsilon)|^2\leq \|f\|_{H^s_{\mathcal{D}}(\mathbb{R}^n)}^2.
    \end{equation}
\end{lemma}
\begin{proof}
    The second inequality in~\eqref{E:seminorm_equivalence} follows directly from the definition of $\|f\|_{H_{\mathcal{D}^s(\mathbb{R}^n)}}$ from~\eqref{E:def_Hs_norm}. For the first inequality, note that
    \begin{align*}
        \sum_{\substack{Q\in\mathcal{D} \\ \varepsilon\not\equiv 1}} |Q|^{-\frac{2s}{n}}|(f,h_Q^\varepsilon)|^2
        &=\sum_{\substack{|Q|<1 \\ \varepsilon\not\equiv 1}} |Q|^{-\frac{2s}{n}}|(f,h_Q^\varepsilon)|^2+\sum_{\substack{|Q|\geq 1 \\ \varepsilon\not\equiv 1}} |Q|^{-\frac{2s}{n}}|(f,h_Q^\varepsilon)|^2\\
        &\leq \sum_{\substack{|Q|<1 \\ \varepsilon\not\equiv 1}} |Q|^{-\frac{2s}{n}}|(f,h_Q^\varepsilon)|^2+\sum_{\substack{Q\in\mathcal{D}\\ \varepsilon\not\equiv 1}} (f,h_Q^\varepsilon)^2\\
        &=\sum_{\substack{|Q|<1 \\ \varepsilon\not\equiv 1}} |Q|^{-\frac{2s}{n}}|(f,h_Q^\varepsilon)|^2+\|f\|_{L^2(\mathbb{R}^n)}^2.
    \end{align*}
\end{proof}

As a consequence of the latter, one recovers the classical relationship between spaces with different regularities.

\begin{remark}\label{R:monotonicity}
    For any $0<s_1<s_2<n$ and $Q\in\mathcal{D}$ with $|Q|\leq 1$ it holds that $|Q|^{-\frac{2s_1}{n}}\leq |Q|^{-\frac{2s_2}{n}}$, whence $H^{s_2}_{\mathcal{D}}(\mathbb{R}^n)\subseteq H^{s_1}_{\mathcal{D}}(\mathbb{R}^n)$.
\end{remark}

From now on and throughout the paper, we will also omit the subscript $\mathcal{D}$ from norms and operators as long as no confusion may occur.

\section{Fractional embeddings} \label{S:embeddings}
This section is devoted to proving the dyadic analogue to the classical fractional embeddings for the Sobolev space $H^s(\mathbb{R}^n)$, see e.g.~\cite[Chapter 7]{Leo23}, now using only dyadic methods. Throughout the section, the expression $A\apprle_{n,s} B$ means that $A$ is bounded by a constant multiple of $B$ and the constant depends on $s$ and $n$. The expression $A\simeq_{n,s} B$ means that $A$ is bounded above and below by a constant multiple of $B$ with a constant that depends on $s$ and $n$.
\subsection{Case \texorpdfstring{$0<s<n/2$}{[small]}}
We start with the analogue to the Sobolev-Gagliardo-Nirenberg embedding, see e.g.~\cite[Theorem 7.6]{Leo23}. To this end, Lemma~\ref{L:dyadic_MS} will allow us to apply known estimates for the operator $T_s$ to one of the classical proofs of the Euclidean fractional Sobolev embedding in the range $0<s<n/2$.
\begin{proposition}
    Let $0<s<\frac{n}{2}$. Then, $\displaystyle H^s(\mathbb{R}^n)\subseteq L^{\frac{2n}{n-2s}}(\mathbb{R}^n)$. In particular,
    \begin{equation*}
    \|f\|_{L^{q}(\mathbb{R}^n)}\apprle_{n,s} \|f\|_{\dot{H}^s(\mathbb{R}^n)}
    \end{equation*}
    for any $f\in \mathscr{S}(\mathbb{R}^n)$ with $q:=\frac{2n}{n-2s}$.
\end{proposition}
\begin{proof}
    Since $\frac{1}{q}=\frac{1}{2}-\frac{s}{n}$, Lemma~\ref{L:dyadic_MS} and the boundedness of the operator $T_s$, c.f.~\cite[Theorem 2.3]{HRS16} or~\cite[Theorem 4]{MW74},
    \[
    \|f\|_{L^q}=(2^s-1)\|T_s(D^{s}f)\|_{L^q}\apprle_{n,s} \|D^{s}f\|_{L^2(\mathbb{R}^n)}=\|f\|_{\dot{H}^s(\mathbb{R}^n)}.
    \]
\end{proof}
\subsection{Case \texorpdfstring{$s=n/2$}{[critical]}}
The classical fractional Sobolev embedding in $\mathbb{R}^n$ is known to fail at the borderline case $s=n/2$ and the proof of that fact readily transfers to our dyadic setting. Besides recording the embedding, in this section we also show that the correct space to embed is ${\rm BMO}_{\mathcal{D}}(\mathbb{R}^n)$, which are bounded mean oscillation functions adapted to the grid $\mathcal{D}$. The proof will use the characterization of the BMO-norm in the dyadic setting given by
\begin{equation}\label{E:BMO_char}
    \|f\|_{{\rm BMO}_{\mathcal{D}}}=\sup_{Q\in\mathcal{D}}\bigg(\frac{1}{|Q|}\sum_{\substack{Q\in\mathcal{D} \\ \varepsilon\not\equiv 1}}|(f,h_Q^\varepsilon)|^2 \bigg)^{1/2}.
\end{equation}

To ease the notation, in the following propositions we will omit the subscript $\mathcal{D}$ and write ${\rm BMO}(\mathbb{R}^n)$ for the dyadic ${\rm BMO}$.
\begin{proposition}
    It holds that $H^{\frac{n}{2}}(\mathbb{R}^n)\subseteq L^q(\mathbb{R}^n)$ for any $1<q<\infty$.
\end{proposition}
\begin{proof}
    After Remark~\ref{R:monotonicity}, the proof can be taken verbatim from~\cite[Theorem 7.12]{Leo23}, setting $p=2$.
\end{proof}

\begin{proposition}
    It holds that $H^{\frac{n}{2}}(\mathbb{R}^n)\subseteq {\rm BMO}(\mathbb{R}^n)$.
\end{proposition}
\begin{proof}
    In view of Lemma~\ref{L:norm_equivalence} and~\eqref{E:BMO_char}, it suffices to prove that for any $f\in\mathscr{S}$ and $Q\in\mathcal{D}$ with $|Q|<1$,
    \[
    \frac{1}{|Q|}\sum_{\substack{P\subseteq Q\\ \varepsilon\not\equiv 1}}|(f,h_P^\varepsilon)|^2\leq \sum_{\substack{P\in\mathcal{D}\\ \varepsilon\not\equiv 1}}|P|^{-1}|(f,h_P^\varepsilon)|^2.
    \]
    However, the latter is clear because
    \[
    \frac{1}{|P|}\sum_{\substack{P\subseteq Q\\ \varepsilon\not\equiv 1}}|(f,h_P^\varepsilon)|^2=\sum_{\substack{P\subseteq Q\\ \varepsilon\not\equiv 1}}\frac{|P|}{|Q|}|P|^{-1}|(f,h_P^\varepsilon)|^2\leq \sum_{\substack{P\subseteq Q\\ \varepsilon\not\equiv 1}}|P|^{-1}|(f,h_P)|^2\leq \sum_{\substack{P\in\mathcal{D}\\ \varepsilon\not\equiv 1}}|P|^{-1}|(f,h_P^\varepsilon)|^2.
    \]
\end{proof}
\subsection{Case \texorpdfstring{$s>n/2$}{[large]}}
The proof of the dyadic Morrey-type embedding in the higher-regularity range follows the ideas of~\cite{ARF24} and is based on the property of the dyadic averages $\langle f\rangle_Q$ proved in the next lemma. For the classical counterpart see~\cite[Theorem 7.23]{Leo23}.

\begin{lemma}\label{L:telescopic_averages}
Let $f\in L^2(\mathbb{R}^n)$, $Q\in\mathcal{D}$ and $k\geq 0$. Then,
    \begin{equation*}
    \langle f\rangle_{Q}-\langle f\rangle_{Q^{(k)}}=\sum_{\substack{Q\subsetneq P\subseteq Q^{(k)} \\ \varepsilon\not\equiv 1}}(f,h^\varepsilon_{P})h_{P}^\varepsilon(Q),
    \end{equation*}
    where $Q^{(k)}$ denotes the $k$th dyadic ancestor of $Q$.
\end{lemma}
\begin{proof}
    Let $k\geq 0$. 
    Applying the Haar expansion of the average in~\eqref{E:avg_char} and noting that $Q\subseteq Q^{(k)}\subsetneq P$ implies $h_P^\varepsilon(Q)=h_P^\varepsilon(Q^{(k)})$, we obtain
    \begin{align*}
    \langle f\rangle_Q&=\sum_{\substack{P\supsetneq Q \\ \varepsilon\not\equiv 1}}(f,h_P^\varepsilon)h_P^\varepsilon(Q)
    =\sum_{\substack{Q^{(k)}\supseteq P\supsetneq Q\\ \varepsilon\not\equiv 1}}(f,h_P^\varepsilon)h_P^\varepsilon(Q)+
    \sum_{\substack{P\supsetneq Q^{(k)}\\ \varepsilon\not\equiv 1}}(f,h_P^\varepsilon)h_P^\varepsilon(Q)\\
    &=\sum_{\substack{Q\subsetneq P\subseteq Q^{(k)} \\ \varepsilon\not\equiv 1}}(f,h^\varepsilon_{P})h_{P}^\varepsilon(Q)+\langle f\rangle_{Q^{(k)}}.
    \end{align*}
\end{proof}
\begin{proposition}\label{P:Morrey}
    Let $\frac{n}{2}<s<n$. Then, $H^s(\mathbb{R}^n)\subseteq L^\infty(\mathbb{R}^n)$ and in particular
    \begin{equation}\label{E:embedding_high}
        \|f\|_{L^\infty(\mathbb{R})}\apprle_{n,s} \|f\|_{\dot{H}^s(\mathbb{R}^n)}
    \end{equation}
    for any $f\in \mathscr{S}(\mathbb{R}^n)$.
\end{proposition}
The exact expression of the constant in the bound is given by~\eqref{E:Morrey_const}.
\begin{proof}
    Let $f\in \mathscr{S}(\mathbb{R}^n)$ and $x\in\mathbb{R}^n$ be a Lebesgue point, i.e.
    \begin{equation}\label{E:Lebesgue_point}
        f(x)=\lim_{k\to\infty}\langle f\rangle_{Q_{x,k}},
    \end{equation}
    where $Q_{x,k}\in \mathcal{D}$ denotes the dyadic cube of side length $2^{-k}$ that contains $x$. Noticing that $|Q_{x,0}|=1$, the triangle inequality and Cauchy-Schwarz yield
    \begin{align*}
        |f(x)|&\leq |f(x)-\langle f\rangle_{Q_{x,0}}|+|\langle f\rangle_{Q_{x,0}}|\\
        &\leq |f(x)-\langle f\rangle_{Q_{x,0}}|+\bigg(\int_{Q_{x,0}}|f|^2dx\bigg)^{1/2}|Q_{x,0}|^{1/2}\\
        &\leq |f(x)-\langle f\rangle_{Q_{x,0}}|+\|f\|_{L^2}.
    \end{align*}
    Therefore, it suffices to show that
    \begin{equation}\label{E:Morrey_02}
        |f(x)-\langle f\rangle_{Q_{x,0}}|\apprle_{n,s}\|f\|_{\dot{H}^s(\mathbb{R}^n)}.
    \end{equation}
    In view of~\eqref{E:Lebesgue_point}, writing the difference on the left-hand side of~\eqref{E:Morrey_02} as a telescopic sum and applying Lemma~\ref{L:telescopic_averages} we get
    \begin{align*}
    |f(x)-\langle f\rangle_{Q_{x,0}}|&\leq \sum_{k=0}^\infty |\langle f\rangle_{Q_{x,k}}-\langle f\rangle_{Q_{x,k+1}}|\\
    &\leq \sum_{k=0}^\infty\; \sum_{\varepsilon\not\equiv 1}|(f,h_{Q_{x,k}}^\varepsilon)|\,|h_{Q_{x,k}}^\varepsilon(Q_{x,k+1})|\\
    &=\sum_{k=0}^\infty\; \sum_{\varepsilon\not\equiv 1} |(f,h_{Q_{x,k}}^\varepsilon)|\,|Q_{x,k}|^{-1/2},
    \end{align*}
    where in the last equality we used the fact that $Q_{x,k+1}\subsetneq Q_{x,k}$ and thus $h_{Q_{x,k}}$ is constant in $Q_{x,k+1}$. Rewriting the above and applying Cauchy-Schwarz we obtain
    \begin{align}
    |f(x)-\langle f\rangle_{Q_{x,0}}|&\leq\sum_{k=0}^\infty \sum_{\varepsilon\not\equiv 1}|(f,h_{Q_{x,k}}^\varepsilon)|\,|Q_{x,k}|^{-\frac{s}{n}}|Q_{x,k}|^{\frac{s}{n}-\frac{1}{2}}\notag\\
    &\leq \bigg(\sum_{k=0}^\infty \sum_{\varepsilon\not\equiv 1}|(f,h_{Q_{x,k}}^\varepsilon)|^2|Q_{x,k}|^{\frac{2s}{n}}\bigg)^{1/2}\bigg(\sum_{k=0}^\infty \sum_{\varepsilon\not\equiv 1}|Q_{x,k}|^{-\frac{2s}{n}-1}\bigg)^{1/2}\notag\\
    &= \bigg(\sum_{k=0}^\infty \sum_{\varepsilon\not\equiv 1}|(f,h_{Q_{x,k}}^\varepsilon)|^2|Q_{x,k}|^{-\frac{2s}{n}}\bigg)^{1/2}\bigg((2^n-1)\sum_{k=0}^\infty(2^{-nk})^{\frac{2s}{n}-1}\bigg)^{1/2}\notag\\
    &\leq \|f\|_{\dot{H}^s}\bigg(\frac{2^n-1}{1-2^{n-2s}}\bigg)^{1/2}\label{E:Morrey_const}
    \end{align}
    which gives~\eqref{E:Morrey_02}.
\end{proof}
\section{Algebra property} \label{S:algebra_prop}
We now turn to the second main question of this paper, that is establishing the validity or failure of the algebra property of the space $H^s(\mathbb{R}^n)$. The question reduces to finding out whether $f\in H^s(\mathbb{R}^n)$ implies $f^2\in H^s(\mathbb{R}^n)$ since if the latter holds and $f,g\in H^s(\mathbb{R}^n)$, then $f+g\in H^s(\mathbb{R}^n)$ and hence $(f+g)^2\in H^s(\mathbb{R}^n)$, whence $fg\in H^s(\mathbb{R}^n)$.

\medskip

The algebra property is tightly connected to the fractional Sobolev embedding, which is instrumental in proving the validity of that property in the higher-regularity range $n/2<s<n$. In the lower-regularity range $0<s\leq n/2$ the space $H^s(\mathbb{R}^n)$ fails to be an algebra, and we provide explicit examples for that failure.

\medskip

Because of the special role played by the function $f^2$ in this section, we record first the following useful expression of its Haar expansion. 
\begin{lemma}\label{L:Haar_coeff_sq}
    For any $f\in\mathscr{S}(\mathbb{R})$ and any Haar function $h_P^\eta$, $P\in\mathcal{D}$, $\eta\not\equiv 1$,
    \begin{equation}\label{E:Haar_coeff_sq}
        (f^2,h_P^\eta) 
        =\sum_{Q\subseteq P}\sum_{\varepsilon\not\equiv 1}(f,h_Q^\varepsilon)^2h_P^\eta(Q)+\sum_{\substack{\varepsilon\not\equiv 1,\widetilde{\varepsilon}\not\equiv 1\\ \varepsilon+\widetilde{\varepsilon}=\eta}}(f,h_Q^\varepsilon)(f,h_Q^{\widetilde{\varepsilon}})+2(f,h_P^\eta)\langle f\rangle_P.
    \end{equation}
\end{lemma}
\begin{proof}
    Writing $f$ in its Haar expansion and applying Lemma~\ref{L:product_Haar},
    \begin{equation}\label{E:f_sq_Haar}
    \begin{aligned}
        f^2&=\bigg(\sum_{Q\in\mathcal{D}}\sum_{\varepsilon\not\equiv 1}(f,h_Q^\varepsilon)h_Q^\varepsilon\bigg)^2\\
        &=\sum_{Q\in\mathcal{D}}\sum_{\widetilde{Q}\in\mathcal{D}}\sum_{\varepsilon\not\equiv 1}\sum_{\widetilde{\varepsilon}\not\equiv 1}(f,h_Q^\varepsilon)(f,h_{\widetilde{Q}}^{\widetilde{\varepsilon}})h_Q^\varepsilon h_{\widetilde{Q}}^{\widetilde{\varepsilon}}\\
        &=\sum_{Q\in\mathcal{D}}\sum_{\varepsilon\not\equiv 1}\sum_{\widetilde{\varepsilon}\not\equiv 1}(f,h_Q^\varepsilon)(f,h_{Q}^{\widetilde{\varepsilon}})h_Q^{\varepsilon+\widetilde{\varepsilon}}
        +2\sum_{Q\in\mathcal{D}}\sum_{\widetilde{Q}\subsetneq Q}\sum_{\varepsilon\not\equiv 1}\sum_{\widetilde{\varepsilon}\not\equiv 1}(f,h_Q^\varepsilon)(f,h_{\widetilde{Q}}^{\widetilde{\varepsilon}})h_Q^\varepsilon h_{\widetilde{Q}}^{\widetilde{\varepsilon}}\\
        &:=S_1+2S_2.
    \end{aligned}
    \end{equation}
    To compute the Haar coefficients of the first term in~\eqref{E:f_sq_Haar}, note that, for any admissible $\eta$ and $P\in\mathcal{D}$,
    \begin{equation*}
        (\mathbf{1}_{Q},h_P^\eta)=
        \begin{cases}
            0&\text{if }P\subsetneq Q,\\
            h_P^\eta(Q)|Q|&\text{if } P\supseteq Q,
        \end{cases}
        \quad\text{and}\quad
        (h_Q^{\varepsilon+\tilde{\varepsilon}},h_P^\eta)=
        \begin{cases}
            1&\text{if }\varepsilon+\tilde{\varepsilon}=\eta\text{ and }P=Q,\\
            0&\text{otherwise.}
        \end{cases}
    \end{equation*}
    Thus, it follows that 
    \begin{equation*}
        (S_1,h_P^\eta)=\sum_{Q\subseteq P}\sum_{\varepsilon\not\equiv 1}(f,h_Q^\varepsilon)^2h_P^\eta(Q)+\sum_{\varepsilon\not\equiv 1}\sum_{\substack{\varepsilon\not\equiv 1\\ \varepsilon+\tilde{\varepsilon}=\eta}}(f,h_Q^\varepsilon)(f,h_q^{\tilde{\varepsilon}}).
    \end{equation*}
    The Haar coefficients of the second term in~\eqref{E:f_sq_Haar} turn out to be expressible in terms of the average of the function after combining Lemma~\ref{L:product_Haar} with the observation
    \begin{equation*}
        (h_Q^\varepsilon,h_P^\eta)=
        \begin{cases}
            1&\text{if }\varepsilon=\eta\text{ and } Q=P,\\
            0&\text{otherwise},
        \end{cases}
    \end{equation*}
    since then it follows that
    \begin{equation*}
        \begin{aligned}
        (S_2,h_P^\eta)&=\sum_{Q\supsetneq P}\sum_{\varepsilon\not\equiv 1}(f,h_Q^\eta)(f,h_P^\eta)h_Q^\varepsilon(P)\\
        &=(f,h_P^\eta)\sum_{Q\supsetneq P}\sum_{\varepsilon\not\equiv 1}(f,h_Q^\eta)h_Q^\varepsilon(P)=(f,h_P^\eta)\langle f\rangle_P.
        \end{aligned}
    \end{equation*}
\end{proof}

\subsection{Case \texorpdfstring{$0<s<n/2$}{[small]}} \label{Counterexample_case<n/2}
To establish the failure of the algebra property in this range we construct a function $f\in H^s(\mathbb{R}^n)$ whose square does not belong to the space. 

\medskip

To do so, let $Q_0:=(-1,0)^n$ and for any $k\geq 1$, let $Q_k:=(-\frac{1}{2^k},0)^n$. In this way, the dyadic child of $Q_{k-1}$ lies in its upper-right corner, see Figure~\ref{F:left_dyadic}. 
Starting from $Q_{0}$ and going up in generations, $Q_{k}$ with $k<0$ denotes the dyadic ancestor of $Q_{k-1}$. In this way, $Q_{k}$ is contained in the far right corner of $Q_{m}$ for any $k\geq 0$ and $m<0$ and hence $h_{Q_{m}}^\epsilon(Q_{k})=1/\sqrt{|Q_{m}|}=2^{mn/2}$ for all $\varepsilon$.

\begin{figure}[H]
    \centering
    \begin{tikzpicture}[scale=8]
        \draw[->] (-1.1, 0) -- (0.1, 0) node[right] {$x$};
        \draw[->] (0, -1.1) -- (0, 0.1) node[above] {$y$};

        \foreach \n in {1,2,4,16} {
            \draw[] (-1/\n, -1/\n) rectangle (0, 0); 
        }

        \filldraw[black] (-1, 0) circle (0.002);
        \node[above left] at (-1,0) {$-1$};

        \filldraw[black] (0, -1) circle (0.002);
        \node[below right] at (0,-1) {$-1$};

        \foreach \n in {2,4,16} {
            \pgfmathsetmacro\x{-1/\n}
            \pgfmathsetmacro\labeln{int(\n)}
            
            \filldraw[black] (\x, 0) circle (0.002);
            \filldraw[black] (0, \x) circle (0.002);
            
            \node[above left] at (\x, 0) {$-\frac{1}{\labeln}$};

            \node[below right] at (0, \x) {$-\frac{1}{\labeln}$};
        }

        \filldraw[black] (0, 0) circle (0.003);
        \node[above right] at (0,0) {$(0,0)$};

        \node[right] at (-3/4,-3/4) {$Q_{0}$};
        \node[right] at (-3/8,-3/8) {$Q_{1}$};
        \node[right] at (-5/32,-5/32) {$Q_{2}$};
    \end{tikzpicture}
    \caption{Dyadic cubes involved in the counterexample~\eqref{E:counterex_low_reg} ($n=2$).}
    \label{F:left_dyadic}
\end{figure}
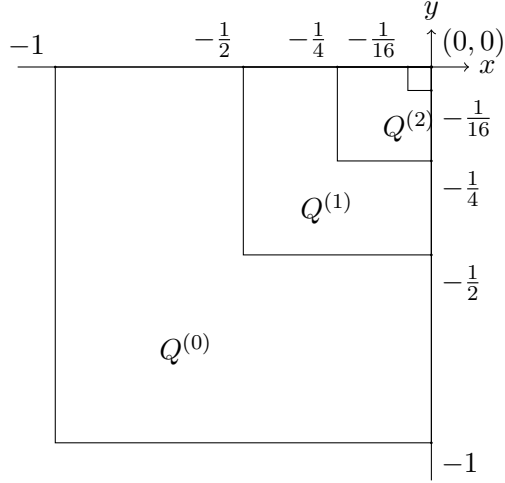
\begin{proposition}\label{P:couterex_low_reg}
    The space $H^s(\mathbb{R}^n)$ is not an algebra when $0<s<n/2$. In particular, let $\varepsilon_0 = (0,1, \ldots, 1)$. The function $f\in\mathscr{S}(\mathbb{R}^n)$ given by
    \begin{equation}\label{E:counterex_low_reg}
        f:=\sum_{k=0}^\infty |Q_{k}|^\alpha \, h^{\varepsilon_0}_{Q_{k}},
    \end{equation}
    where $\frac{s}{n}<\alpha<\frac{s}{2n}+\frac{1}{4}$, belongs to $H^s(\mathbb{R}^n)$ while $f^2\notin H^s(\mathbb{R}^n)$.
\end{proposition}

\begin{remark}
Note that, if $0<s<\frac{n}{2}$, then $\frac{s}{n}<\frac{s}{2n}+\frac{1}{4}$ and hence $\alpha$ can be chosen in the desired range.
\end{remark}

\begin{proof}
    We first prove that the function $f$ in~\eqref{E:counterex_low_reg} belongs to $H^s(\mathbb{R})$. By definition, its Haar coefficients are given by
    \begin{equation*}
        (f,h_Q^{\varepsilon})=\begin{cases}
            |Q_{k}|^\alpha=2^{-kn\alpha}&\text{if }Q=Q_{k}\text{ for some }k \geq 0, \text{ and } \varepsilon=\varepsilon_0,\\
            0&\text{otherwise}.
        \end{cases}
    \end{equation*}
    Thus,
    \begin{align*}
    \| D^sf \|_{L^2}^2 
    & = \sum_{Q\in \mathcal{D}} \sum_{\varepsilon \not\equiv 1} |Q|^{-\frac{2s}{n}} |(f,h_Q^\varepsilon)|^2 
    =\sum_{Q\in \mathcal{D}} |Q|^{-\frac{2s}{n}} |(f,h_Q^{\varepsilon_0})|^2 \\
    &= \sum_{k\geq 0} |Q_{k}|^{-\frac{2s}{n}} |(f,h_{Q_{k}}^{\varepsilon_0})|^2 
    = \sum_{k\geq 0} 2^{2ks} \, 2^{-2kn\alpha}
    = \sum_{k\geq 0} 2^{2k(s-n\alpha)}
    \end{align*}
    and the latter converges because $\alpha>\frac{s}{n}$.
    To analyze $f^2$, notice that since $\varepsilon_0$ is fixed, the Haar expansion of $f$ does not involve any summation over $\varepsilon$ and therefore 
    \begin{equation*}
        f=\sum_{Q\in \mathcal{D}} \sum_{\varepsilon \not\equiv 1} (f,h_Q^\varepsilon) h_Q^\varepsilon f
        =\sum_{Q\in \mathcal{D}} (f,h_Q^{\varepsilon_0}) h_Q^{\varepsilon_0}
        = \sum_{k>0} (f,h_{Q_{k}}^{\varepsilon_0}) h_{Q_{k}}^{\varepsilon_0}.
    \end{equation*}
    Taking the square of the latter,
    \begin{align}
        f^2 & = \sum_{k>0} (f,h_{Q_{k}}^{\varepsilon_0})^2 (h_{Q_{k}}^{\varepsilon_0})^2 + 2 \sum_{m>0}\sum_{k>m}(f,h_{Q_{k}}^{\varepsilon_0})(f,h_{Q_{m}}^{\varepsilon_0}) h_{Q_{k}}^{\varepsilon_0}h_{Q_{m}}^{\varepsilon_0}\notag\\
        & = \sum_{k>0} (f,h_{Q_{k}}^{\varepsilon_0})^2 \frac{1}{|Q_{k}|}\mathbf{1}_{Q_{k}} + 2 \sum_{m>0} \sum_{k>m}(f,h_{Q_{k}}^{\varepsilon_0})(f,h_{Q_{m}}^{\varepsilon_0}) h_{Q_{k}}^{\varepsilon_0}h_{Q_{m}}^{\varepsilon_0}(Q_{k}).
        \label{E:Haar_exp_counterexample}
    \end{align}
    
    Recall that $\frac{1}{|Q_{k}|}\mathbf{1}_{Q_{k}}=\sum_{P \supseteq Q_{k}} \sum_{\varepsilon \not\equiv 1} h_P^\varepsilon (Q_{k}) h_P^\varepsilon$. By definition of $Q_{k}$, $P$ must equal $Q_{l}$ for some $l<k$. In addition, as noticed at the beginning of section~\ref{Counterexample_case<n/2}, $h_P^\varepsilon (Q_{k}) \equiv \frac{1}{\sqrt{|P|}}$, which is positive for every $\varepsilon$ and every $P \supseteq Q_{k}$. Thus, it follows from~\eqref{E:Haar_exp_counterexample} that
    {\small
    \begin{align*}
        f^2 & 
         = \sum_{k>0} (f,h_{Q_{k}}^{\varepsilon_0})^2 \sum_{m<k} \sum_{\varepsilon \not\equiv 1} h_{Q_{m}}^{\varepsilon}(Q_{k}) h_{Q_{m}}^{\varepsilon} +  2 \sum_{m>0} \sum_{k>m}(f,h_{Q_{k}}^{\varepsilon_0})(f,h_{Q_{m}}^{\varepsilon_0}) h_{Q_{m}}^{\varepsilon_0}(Q_{k}) h_{Q_{k}}^{\varepsilon_0} \\
        & = \sum_{m \in \mathbb{Z}} \sum_{k>\max\{0,m\}}(f,h_{Q_{k}}^{\varepsilon_0})^2 \sum_{\varepsilon \not\equiv 1} \frac{1}{\sqrt{|Q_{m}|}} h_{Q_{m}}^{\varepsilon} + 2 \sum_{k>0} \sum_{0<m<k} (f,h_{Q_{k}}^{\varepsilon_0})(f,h_{Q_{m}}^{\varepsilon_0}) \frac{1}{\sqrt{|Q_{m}|}} h_{Q_{k}}^{\varepsilon_0} \\
        & = \sum_{m>0} \left( \sum_{k>m} |Q_{k}|^{2\alpha} \right)  2^{\frac{mn}{2}}
 \sum_{\varepsilon \not\equiv 1} h_{Q_{m}}^\varepsilon + \sum_{m \leq 0 }   \left( \sum_{k>0} |Q_{k}|^{2\alpha} \right) 2^{\frac{mn}{2}}\sum_{\varepsilon \not\equiv 1} h_{Q_{m}}^\varepsilon  \\
        & + 2 \sum_{k>0} |Q_{k}|^{\alpha} \left( \sum_{0<m<k} |Q_{m}|^{\alpha} 2^{\frac{mn}{2}} \right) h_{Q_{k}}^{\varepsilon_0} \\
        & = \sum_{m>0} \left[ 2^{\frac{mn}{2}} \left( \sum_{k>m} 2^{-2kn\alpha} \right) \sum_{\varepsilon \not\equiv 1} h_{Q_{m}}^\varepsilon + 2^{-mn\alpha} \left( \sum_{0<k<m} 2^{-kn\alpha} 2^{\frac{kn}{2}} \right)  h_{Q_{m}}^{\varepsilon_0} \right] \\
        &+ C \sum_{m\leq 0} 2^{\frac{mn}{2}} \sum_{\varepsilon \not\equiv 1} h_{Q_{m}}^\varepsilon \\
        & = \sum_{m>0} \left[ 2^{\frac{mn}{2}} C_m  + 2^{-mn\alpha} \widetilde{C}_m   \right]  h_{Q_{m}}^{\varepsilon_0} + \sum_{m>0} 2^{\frac{mn}{2}} C_m \sum_{\substack{\varepsilon \not\equiv 1 \\ \varepsilon \neq \epsilon_0}} h_{Q_{m}}^\varepsilon + C \sum_{m\leq 0} 2^{\frac{mn}{2}} \sum_{\varepsilon \not\equiv 1} h_{Q_{m}}^\varepsilon,
    \end{align*}
    }
        where in the last equality, 
        $ C_m :=  \sum_{k>m} 2^{-2kn\alpha}$ and $ \widetilde{C}_m  :=  \sum_{0<k<m} 2^{-kn\alpha+\frac{kn}{2}}$. Therefore, 
    \begin{equation}\label{E:Haar_coeffs_counter_sq}
            (f^2,h_Q^\varepsilon)=\begin{cases}
                C 2^\frac{mn}{2}&\text{if }Q=Q_{m}\text{ for some } m\leq 0, \text{ and } \varepsilon \not\equiv 1,\\
                C_m 2^\frac{mn}{2} &\text{if }Q=Q_{m}\text{ for some } m>0, \text{ and } \varepsilon \not\equiv 1, \varepsilon \neq \varepsilon_0\\
                C_m 2^\frac{mn}{2} +  2^{-mn\alpha} \widetilde{C}_m &\text{if }Q=Q_{m}\text{ for some } m > 0,\text{ and } \varepsilon= \varepsilon_0\\
                0&\text{else}
            \end{cases}
        \end{equation}
    and
    \begin{align*}
        \widetilde{C}_m  &:=  \sum_{0<k<m} 2^{-kn\alpha+\frac{kn}{2}} = \sum_{0<k<m} (2^{\frac{n}{2}-n\alpha})^k = \frac{(2^{\frac{n}{2}-n\alpha})^m-2^{\frac{n}{2}-n\alpha}}{2^{\frac{n}{2}-n\alpha}-1} \sim (2^{\frac{n}{2}-n\alpha})^m.
    \end{align*}
    Plugging the coefficients~\eqref{E:Haar_coeffs_counter_sq} into the Sobolev seminorm yields the explicit lower bound 
    
    \begin{align*}
        \| D^s f^2 \|_{L^2}^2 &\geq \sum_{m>0} |Q_{m}|^{-\frac{2s}{n}}|(f^2,h_{Q_{m}}^{\varepsilon_0})|^2 = \sum_{m>0} (2^{-mn})^{-\frac{2s}{n}} (C_m 2^\frac{mn}{2} + 2^{-mn\alpha} \widetilde{C}_m )^2 \\
        & \geq  \sum_{m>0} 2^{2ms} 2^{-2mn\alpha} \widetilde{C}_m ^2 =  \sum_{m>0} 2^{2ms}\, 2^{-2mn\alpha} \, 2^{m(n-2n\alpha)} \\
        &= \sum_{m>0} 2^{m(2s-4n\alpha+n)}.
    \end{align*}
    The latter series diverges because $\alpha < \frac{s}{2n}+\frac{1}{4}$, hence $f^2 \notin H^s(\mathbb{R}^n)$ as desired.
\end{proof}

\subsection{Case \texorpdfstring{$s=1/2$}{[critical]}}
In the critical case, the space $H^{\frac{n}{2}}(\mathbb{R}^n)$ also fails to be an algebra. Instead of a function like~\eqref{E:counterex_low_reg}, we see in Proposition~\ref{P:counterex_BMO} that the counterexample requires logarithmic correction in the weights.

\begin{proposition}\label{P:counterex_BMO}
    The space $H^{\frac{n}{2}}(\mathbb{R}^n)$ is not an algebra. In particular, let $\varepsilon_0 := (0,1,...,1)$. Then, the function
    \[
    f:= \sum_{k>0} \frac{1}{2^{kn/2}k^{\alpha/2}} h_{Q_{k}}^{\varepsilon_0}
    \]
    with $1<\alpha \leq 3/2$ belongs to $H^{\frac{n}{2}}(\mathbb{R}^n)$ while $f^2\notin H^{\frac{n}{2}}(\mathbb{R}^n)$.
\end{proposition}

\begin{proof}
    Note first that, by definition of $f$, for any $Q\in\mathcal{D}$
    \begin{equation*}
        (f,h_Q^\varepsilon)
        =\begin{cases}
            \frac{1}{2^{kn/2}k^{\alpha/2}} &\text{if }Q=Q_{k}\text{ for some }k> 0 \text{ and } \varepsilon=\varepsilon_0,\\
            0&\text{otherwise},
        \end{cases}
    \end{equation*}
whence
    \begin{align*}
    \| D^{\frac{n}{2}}f \|_{L^2}^2 
    &=  \sum_{Q\in \mathcal{D}} \sum_{\varepsilon \not\equiv 1} |Q|^{-1} |(f,h_Q^\varepsilon)|^2 = \sum_{Q\in \mathcal{D}} |Q|^{-1} |(f,h_Q^{\varepsilon_0})|^2
     = \sum_{k> 0} |Q_{k}|^{-1} |(f,h_{Q_{k}}^{\varepsilon_0})|^2 \\
    & = \sum_{k> 0} 2^{kn} \, \frac{1}{2^{kn}k^{\alpha}}= \sum_{k> 0} \frac{1}{k^\alpha}
    \end{align*}
and the latter converges since $\alpha>1$. A similar computation gives
    \begin{equation*}
    \| f \|_{L^2}^2 =  \sum_{Q\in \mathcal{D}} |(f,h_Q^{\varepsilon_0})|^2 = \sum_{k> 0} \frac{1}{2^{kn}k^{\alpha}} < +\infty,
    \end{equation*}
    hence $f\in H^{\frac{n}{2}}(\mathbb{R}^n)$. 

    To see that $f^2\notin H^{n/2}(\mathbb{R}^n)$ we first need to find the Haar coefficients of $f^2$. So, let's denote $\lambda_k := (f,h_{Q_{k}}^{\epsilon_0})$ for simplicity 
    and, similarly to the low regularity case, we use \eqref{E:Haar_exp_counterexample} to write
    \begin{align*}
        f= \sum_{k>0} (f,h_{Q_{k}}^{\varepsilon_0}) h_{Q_{k}}^{\varepsilon_0}(x).
    \end{align*}
    Taking the square, as in Proposition \ref{P:couterex_low_reg}, we get
    \begin{align*}
        f^2 & = \sum_{k>0} (f,h_{Q_{k}}^{\epsilon_0})^2 (h_{Q_{k}}^{\epsilon_0})^2 + 2 \sum_{m>0}\sum_{k>m}(f,h_{Q_{k}}^{\epsilon_0})(f,h_{Q_{m}}^{\epsilon_0}) h_{Q_{k}}^{\epsilon_0}h_{Q_{m}}^{\epsilon_0}\\
        & = \sum_{m \in \mathbb{Z}} \sum_{k>\max\{0,m\}}(f,h_{Q_{k}}^{\epsilon_0})^2 \sum_{\epsilon \not\equiv 1} \frac{1}{\sqrt{|Q_{m}|}} h_{Q_{m}}^{\epsilon} + 2 \sum_{k>0} \sum_{0<m<k} (f,h_{Q_{k}}^{\epsilon_0})(f,h_{Q_{m}}^{\epsilon_0}) \frac{1}{\sqrt{|Q_{m}|}} h_{Q_{k}}^{\epsilon_0} \\
        & = \sum_{m>0} \left( \sum_{k>m} \lambda_k^2 \right)  2^{\frac{mn}{2}}
 \sum_{\epsilon \not\equiv 1} h_{Q_{m}}^\epsilon + \sum_{m \leq 0 }   \left( \sum_{k>0} \lambda_k^2 \right) 2^{\frac{mn}{2}}\sum_{\epsilon \not\equiv 1} h_{Q_{m}}^\epsilon  \\
        & + 2 \sum_{k>0} \lambda_k \left( \sum_{0<m<k} \lambda_m 2^{\frac{mn}{2}} \right) h_{Q_{k}}^{\epsilon_0} \\
        & = \sum_{m>0} \left[ 2^{\frac{mn}{2}} \left( \sum_{k>m} \lambda_k^2 \right) \sum_{\epsilon \not\equiv 1} h_{Q_{m}}^\epsilon + \lambda_m \left( \sum_{0<k<m} \lambda_k 2^{\frac{kn}{2}} \right)  h_{Q_{m}}^{\epsilon_0} \right]  + C \sum_{m\leq 0} 2^{\frac{mn}{2}} \sum_{\epsilon \not\equiv 1} h_{Q_{m}}^\epsilon \\
        & = \sum_{m>0} \left[ 2^{\frac{mn}{2}} C_m  + \lambda_m \widetilde{C}_m   \right]  h_{Q_{m}}^{\epsilon_0} + \sum_{m>0} 2^{\frac{mn}{2}} C_m \sum_{\substack{\epsilon \not\equiv 1 \\ \epsilon \neq \epsilon_0}} h_{Q_{m}}^\epsilon + C \sum_{m\leq 0} 2^{\frac{mn}{2}} \sum_{\epsilon \not\equiv 1} h_{Q_{m}}^\epsilon,
    \end{align*}
    where in the third equality, we break the first term into two sums for $m>0$ and $m\leq 0$ and remember that $\lambda_k = (f,h_{Q_{k}}^{\epsilon_0})$. Next, we rename $k$ to $m$ and $m$ to $k$ on the third term above, gather the terms together and notice that the series of $\lambda_k^2$ converges. Finally, the last equality follows by isolating $h_{Q_{m}}^{\epsilon_0}$ from $ \sum_{\epsilon \not\equiv 1} h_{Q_{m}}^\epsilon$ and denoting $ C_m :=  \sum_{k>m} \lambda_k^2$ and $ \widetilde{C}_m  :=  \sum_{0<k<m} \lambda_k 2^{\frac{kn}{2}}$. Therefore, we have that
    \begin{equation*}
            (f^2,h_Q^\epsilon)=\begin{cases}
                C 2^\frac{mn}{2}&\text{if }Q=Q_{m}\text{ for some } m\leq 0, \text{ and } \epsilon \not\equiv 1,\\
                C_m 2^\frac{mn}{2} &\text{if }Q=Q_{m}\text{ for some } m>0, \text{ and } \epsilon \not\equiv 1, \epsilon \neq \epsilon_0\\
                C_m 2^\frac{mn}{2} + \lambda_m \widetilde{C}_m &\text{if }Q=Q_{m}\text{ for some } m > 0,\text{ and } \epsilon= \epsilon_0\\
                0&\text{else}.
            \end{cases}
        \end{equation*}
    We notice also that 
    \begin{align*}
        \widetilde{C}_m  &:=  \sum_{0<k<m} \lambda_k 2^{\frac{kn}{2}} = \sum_{0<k<m} \frac{1}{k^{\alpha/2}} \sim \int_1^m x^{-\alpha/2} dx \sim m^{-\frac{\alpha}{2}+1}.
    \end{align*}
    Thus, 
    \begin{align*}
        \| D^{\frac{n}{2}} f^2 \|_{L^2}^2 &\geq \sum_{m>0} |Q_{m}|^{-1}|(f^2,h_{Q_{m}}^{\epsilon_0})|^2 = \sum_{m>0} 2^{mn} (C_m 2^\frac{mn}{2} + \lambda_m \widetilde{C}_m )^2 \\
        & \geq  \sum_{m>0} 2^{mn} \lambda_m^2 \widetilde{C}_m ^2 = \sum_{m>0} \frac{1}{m^{2\alpha -2}}
    \end{align*}
    and the latter series diverges because $\alpha \leq \frac{3}{2}$. Therefore, $f^2 \notin H^{\frac{n}{2}}(\mathbb{R}^n)$.
\end{proof}

\subsection{High regularity: \texorpdfstring{$n/2<s<n$}{[large]}}\mbox{}
Besides the Sobolev embedding~\eqref{E:embedding_high}, the key estimate in the proof of the algebra property in the ``high regularity range'' is an estimate that bounds an expression involving $f^2$ by another which, instead, depends on $f$ without powers, c.f. Lemma~\ref{L:estimate_high_reg}. While being a generalization of the 1-dimensional estimate obtained in~\cite[Lemma 12]{ARF24}, the proof of the latter in $n$ dimensions is substantially more involved and relies on an observation of independent interest that is presented separately in Lemma~\ref{L:estimate_high_reg_help}.

\begin{lemma}\label{L:estimate_high_reg}
    For any fixed $Q\in\mathcal{D}$ with $|Q|\leq 1$,
    \begin{equation}\label{E:estimate_high_reg}
        \sum_{P\subseteq Q}\sum_{\varepsilon\not\equiv 1}|P|^{-\frac{2s}{n}}|(f^2,h_P^\varepsilon)|^2\apprle_{s,n} |Q|^{\frac{2s}{n}-1}\|f\|_{\dot{H}^s}^2\sum_{P\subseteq Q}\sum_{\varepsilon\not\equiv 1}|P|^{-\frac{2s}{n}}|(f,h_P^\varepsilon)|^2.
    \end{equation}
\end{lemma}

\begin{proof}
    Substituting the Haar coefficients by their expression from~\eqref{E:Haar_coeff_sq}, the right-hand side of~\eqref{E:estimate_high_reg} can be estimated as
        \begin{align}
            &\sum_{P\subseteq Q}\sum_{\varepsilon\not\equiv 1}|P|^{-\frac{2s}{n}}|(f^2,h_P^\varepsilon)|^2 \notag \\
            &=\sum_{P\subseteq Q}\sum_{\varepsilon\not\equiv 1}|P|^{-\frac{2s}{n}}\bigg(\sum_{\widetilde{Q}\subseteq P}\sum_{\tilde{\varepsilon}\not\equiv 1}(f,h_{\widetilde{Q}}^{\tilde{\varepsilon}})^2h_p^\varepsilon(\widetilde{Q})+\sum_{\substack{\varepsilon_1,\varepsilon_2\not\equiv 1\\ \varepsilon_1+\varepsilon_2=\varepsilon}}(f,h_P^{\varepsilon_1})(f,h_P^{\varepsilon_2})+2(f,h_p^\varepsilon)\langle f\rangle_P\bigg)^2 \notag \\
            &\leq 4\sum_{P\subseteq Q}\sum_{\varepsilon\not\equiv 1}|P|^{-\frac{2s}{n}}\bigg(\sum_{\tilde{Q}\subseteq P}\sum_{\tilde{\varepsilon}\not\equiv 1}(f,h_{\tilde{Q}}^{\tilde{\varepsilon}})^2h_P^\varepsilon(\tilde{Q})\bigg)^2 \label{E:high_reg_01} \\
            &+4\sum_{P\subseteq Q}\sum_{\varepsilon\not\equiv 1}|P|^{-\frac{2s}{n}}\bigg(\sum_{\substack{\varepsilon_1,\varepsilon_2\not\equiv 1\\ \varepsilon_1+\varepsilon_2=\varepsilon}}(f,h_P^{\varepsilon_1})(f,h_P^{\varepsilon_2})\bigg)^2 \notag \\
            &+4\sum_{P\subseteq Q}\sum_{\varepsilon\not\equiv 1}|P|^{-\frac{2s}{n}}(f,h_P^\varepsilon)^2\langle f\rangle_P^2
            =:4(S_1+S_2+S_3).\notag
        \end{align}
    To estimate $S_3$, applying the embedding from Proposition~\ref{P:Morrey} yields
    \begin{align*}
       \sum_{P\subseteq Q}\sum_{\varepsilon\not\equiv 1}|P|^{-\frac{2s}{n}}(f,h_P^\varepsilon)^2\langle f\rangle_P^2
       &\leq \|f\|_{L^\infty}^2\sum_{P\subseteq Q}\sum_{\varepsilon\not\equiv 1}|P|^{-\frac{2s}{n}}(f,h_P^\varepsilon)^2\\
       &\apprle_{s,n}\|f\|_{\dot{H}^s(\mathbb{R}^n)}^2\sum_{P\subseteq Q}\sum_{\varepsilon\not\equiv 1}|P|^{-\frac{2s}{n}}(f,h_P^\varepsilon)^2.
    \end{align*}
    To estimate $S_2$, note first that the definition of summation for indices $\varepsilon$, c.f.~\eqref{E:def_index_sum}, the condition $\varepsilon_1+\varepsilon_2=\varepsilon$ is equivalent to $\varepsilon_1+\varepsilon=\varepsilon_2$ and we may thus rewrite 
    \begin{equation*}
        \sum_{P\subseteq Q}\sum_{\varepsilon\not\equiv 1}|P|^{-\frac{2s}{n}}\bigg[\sum_{\substack{\varepsilon_1,\varepsilon_2\not\equiv 1\\ \varepsilon_1+\varepsilon_2=\varepsilon}}(f,h_P^{\varepsilon_1})(f,h_P^{\varepsilon_2})\bigg]^2
        =\sum_{P\subseteq Q}\sum_{\varepsilon\not\equiv 1}|P|^{-\frac{2s}{n}}\bigg[\sum_{\substack{\varepsilon_1\not\equiv 1\\ \varepsilon_1\not\equiv\varepsilon}}(f,h_P^{\varepsilon_1})(f,h_P^{\varepsilon+\varepsilon_1})\bigg]^2.
    \end{equation*}
    Applying Cauchy-Schwarz and Lemma~\ref{L:estimate_high_reg_help} the latter is bounded by
    \begin{align*}
      &\sum_{P\subseteq Q}\sum_{\varepsilon\not\equiv 1}|P|^{-\frac{2s}{n}}\bigg[\sum_{\substack{\varepsilon_1\not\equiv 1\\ \varepsilon_1\not\equiv\varepsilon}}(f,h_P^{\varepsilon_1})^2\bigg]
      \bigg[\sum_{\substack{\varepsilon_1\not\equiv 1\\ \varepsilon_1\not\equiv\varepsilon}}(f,h_P^{\varepsilon+\varepsilon_1})^2\bigg]\\
      &\leq \sum_{P\subseteq Q}|P|^{-\frac{2s}{n}}\bigg[\sum_{\tilde{\varepsilon}\not\equiv 1}(f,h_P^{\tilde{\varepsilon}})^2\bigg]\bigg[\sum_{\varepsilon\not\equiv 1}\sum_{\substack{\varepsilon_1\not\equiv 1\\ \varepsilon_1\not\equiv\varepsilon}}(f,h_P^{\varepsilon+\varepsilon_1})^2\bigg]\\
      &\apprle_n \|f\|_{H^s(\mathbb{R}^n)}^2 \sum_{P\subseteq Q}|P|^{-\frac{2s}{n}}\sum_{\tilde{\varepsilon}\not\equiv 1}(f,h_P^{\tilde{\varepsilon}})^2.
    \end{align*}
    Finally, we estimate $S_1$ by first writing out the square in~\eqref{E:high_reg_01}. Then,
    \begin{multline}\label{E:high_reg_02}
        \sum_{P\subseteq Q}\sum_{\varepsilon\not\equiv 1}|P|^{-\frac{2s}{n}}
        \sum_{\widetilde{Q}\subsetneq P}\sum_{\widetilde{K}\subsetneq P}\sum_{\tilde{\varepsilon}\not\equiv 1}\sum_{\tilde{\eta}\not\equiv1}
        (f,h_{\widetilde{Q}}^{\tilde{\varepsilon}})^2(f,h_{\widetilde{K}}^{\tilde{\eta}})^2h_P^\varepsilon(\widetilde{Q})h_P^\varepsilon(\widetilde{K})\\
        =\sum_{\widetilde{Q}\subsetneq Q}\sum_{\tilde{\varepsilon}\not\equiv 1}|\widetilde{Q}|^{-\frac{2s}{n}}(f,h_{\widetilde{Q}}^{\tilde{\varepsilon}})^2\sum_{\widetilde{K}\subsetneq Q}\sum_{\tilde{\eta}\not\equiv1}|\widetilde{K}|^{-\frac{2s}{n}}(f,h_{\widetilde{K}}^{\tilde{\eta}})^2\times\\
        \times
        \bigg[\sum_{\substack{\widetilde{Q}\subsetneq P\subseteq Q\\ \widetilde{K}\subsetneq P\subseteq Q}}\sum_{\varepsilon\not\equiv1}|P|^{-\frac{2s}{n}} |\widetilde{Q}|^{\frac{2s}{n}}h_P^\varepsilon(\widetilde{Q})|\widetilde{K}|^{\frac{2s}{n}}h_P^\varepsilon(\widetilde{K})\bigg].
    \end{multline}
    Observe that it is enough to show that the term in brackets is bounded by a constant multiple of $|Q|^{\frac{2s}{n}-1}$. Afterwards, it will follow that 
    \begin{align*}
        S_1&\apprle_{s,n} |Q|^{\frac{2s}{n}-1}\sum_{\widetilde{Q}\subsetneq P}\sum_{\tilde{\varepsilon}\not\equiv 1}|\widetilde{Q}|^{-\frac{2s}{n}}(f,h_{\widetilde{Q}}^{\tilde{\varepsilon}})^2\sum_{\widetilde{K}\subsetneq P}\sum_{\tilde{\eta}\not\equiv1}|\widetilde{K}|^{-\frac{2s}{n}}(f,h_{\widetilde{K}}^{\tilde{\eta}})^2\\
        &\leq |Q|^{\frac{2s}{n}-1}\|f\|_{\dot{H}^s}^2\sum_{\widetilde{K}\subsetneq P}\sum_{\tilde{\eta}\not\equiv1}|\widetilde{K}|^{-\frac{2s}{n}}(f,h_{\widetilde{K}}^{\tilde{\eta}})^2.
    \end{align*}
    To estimate the bracketed term in~\eqref{E:high_reg_02}, note first that $\widetilde{Q}\subseteq P$, whence $h_P^\varepsilon(\widetilde{Q})=\pm|P|^{-1/2}\leq |P|^{-1/2}$ and the same holds for $\widetilde{K}\subseteq P$. Moreover, the number of admissible indices $\varepsilon$ is $2^n-1$. Thus,
    \begin{align*}
       &\sum_{\substack{\widetilde{Q}\subsetneq P\subseteq Q\\ \widetilde{K}\subsetneq P\subseteq Q}}\sum_{\varepsilon\not\equiv1}|P|^{-\frac{2s}{n}} |\widetilde{Q}|^{\frac{2s}{n}}h_P^\varepsilon(\widetilde{Q})|\widetilde{K}|^{\frac{2s}{n}}h_P^\varepsilon(\widetilde{K})
       \leq \sum_{\substack{\widetilde{Q}\subsetneq P\subseteq Q\\ \widetilde{K}\subsetneq P\subseteq Q}}\sum_{\varepsilon\not\equiv1}|P|^{-\frac{2s}{n}-1} |\widetilde{Q}|^{\frac{2s}{n}}|\widetilde{K}|^{\frac{2s}{n}}\\
       &\leq (2^n-1)\!\!\!\sum_{\substack{\widetilde{Q}\subsetneq P\subseteq Q\\ |\widetilde{K}|\leq |\widetilde{Q}|}}|P|^{-\frac{2s}{n}-1} |\widetilde{Q}|^{\frac{2s}{n}}|\widetilde{K}|^{\frac{2s}{n}}
       +(2^n-1)\!\!\!\sum_{\substack{\widetilde{K}\subsetneq P\subseteq Q\\ |\widetilde{Q}|\leq |\widetilde{K}|}}|P|^{-\frac{2s}{n}-1} |\widetilde{Q}|^{\frac{2s}{n}}|\widetilde{K}|^{\frac{2s}{n}}\\
       &\leq 2(2^n-1)\!\!\!\sum_{\widetilde{Q}\subsetneq P\subseteq Q}|P|^{-\frac{2s}{n}-1} |\widetilde{Q}|^{\frac{4s}{n}}
       \leq 2(2^n-1)\!\!\!\sum_{\widetilde{Q}\subsetneq P\subseteq Q}\Big(\frac{|\widetilde{Q}|}{|P|}\Big)^{\frac{2s}{n}} |P|^{\frac{2s}{n}-1}\\
       &<2^n|Q|^{\frac{2s}{n}-1}\sum_{\widetilde{Q}\subsetneq P\subseteq Q}\Big(\frac{|\widetilde{Q}|}{|P|}\Big)^{\frac{2s}{n}}
       \apprle_{s,n}|Q|^{\frac{2s}{n}-1},
    \end{align*}
    where in the last line we use the fact that $\frac{2s}{n}>1$ and $|\widetilde{Q}|<|P|<|Q|$.
\end{proof}

\begin{lemma}\label{L:estimate_high_reg_help}
    Let $s>\frac{n}{2}$. For any $P\in\mathcal{D}$ with $|P|<1$,
    \begin{equation}\label{E:estimate_high_reg_help}
        \sum_{\varepsilon\not\equiv1}\sum_{\substack{\tilde{\varepsilon}\not\equiv1\\ \tilde{\varepsilon}\not\equiv\varepsilon}}(f,h_P^{\varepsilon+\tilde{\varepsilon}})^2\leq 2^n\|f\|_{\dot{H}^s(\mathbb{R}^n)}^2.
    \end{equation}
\end{lemma}

\begin{proof}
    Note that the left hand side of~\eqref{E:estimate_high_reg_help} may alternatively be written as 
    \[
    \sum_{\varepsilon\not\equiv1}\sum_{\substack{\tilde{\varepsilon}\not\equiv1\\ \tilde{\varepsilon}\not\equiv\varepsilon}}(f,h_P^{\tilde{\varepsilon}})^2.
    \]
    Moreover, since $|P|<1$ and $\frac{2s}{n}>1$, it follows that $1\leq |P|^{-\frac{2s}{n}}$ and therefore the expression above is bounded by
    \begin{equation*}
        \#\{\varepsilon\not\equiv1\}\sum_{\tilde{\varepsilon}\not\equiv1}(f,h_P^{\tilde{\varepsilon}})^2
        \leq 2^n\sum_{\tilde{\varepsilon}\not\equiv1}(f,h_P^{\tilde{\varepsilon}})^2
        \leq 2^n\sum_{\tilde{\varepsilon}\not\equiv1}|P|^{-\frac{2s}{n}}(f,h_P^{\tilde{\varepsilon}})^2\leq 2^n\|f\|_{\dot{H}^s(\mathbb{R}^n)}
    \end{equation*}
    as we wanted to prove.
\end{proof}

With the previous lemmata on hand, we can now prove the algebra property of $H^s(\mathbb{R})$ in the high regularity range.

\begin{proposition}
    Let $\frac{n}{2}<s<1$. Then, $f\in H^s(\mathbb{R}^n)$ implies $f^2\in H^s(\mathbb{R}^n)$.
\end{proposition}
\begin{proof}
    Let $\mathcal{D}_1$ be a partition of $\mathbb{R}^n$ by dyadic intervals of length $2^{-1}$. By virtue of Lemma~\ref{L:norm_equivalence}, Lemma~\ref{L:estimate_high_reg} and Proposition~\ref{P:Morrey} we obtain
    \begin{align*}
        \|f^2\|_{H^s(\mathbb{R}^n)}^2&\leq 2\|f^2\|_{L^2(\mathbb{R}^n)}^2+2\sum_{|Q|<1}\sum_{\varepsilon\not\equiv 1}|Q|^{-\frac{2s}{n}}(f^2,h_Q^\varepsilon)|^2\\ 
        &=2\|f^2\|_{L^2(\mathbb{R}^n)}^2+2\sum_{Q\in\mathcal{D}_1}\sum_{P\subseteq Q}\sum_{\varepsilon\not\equiv 1}|P|^{-\frac{2s}{n}}|(f^2,h_P^\varepsilon)|^2\\
        &\apprle_{n,s} \|f\|_{L^\infty(\mathbb{R}^n)}^2\|f\|_{L^2(\mathbb{R}^n)}^2+\|f\|_{\dot{H}^s}^2\sum_{Q\in\mathcal{D}_1}|Q|^{\frac{2s}{n}-1}\sum_{P\subseteq Q}\sum_{\varepsilon\not\equiv 1}|P|^{-\frac{2s}{n}}|(f,h_P^\varepsilon)|^2\\
        &\apprle_{n,s} \|f\|_{H^s(\mathbb{R}^n)}^2\bigg(\|f\|_{L^2(\mathbb{R}^n)}^2+2^{1-\frac{2s}{n}}\sum_{Q\in\mathcal{D}_1}\sum_{P\subseteq Q}\sum_{\varepsilon\not\equiv 1}|P|^{-\frac{2s}{n}}|(f,h_P^\varepsilon)|^2\bigg)\\
        &\apprle_{n,s} \|f\|_{H^s(\mathbb{R}^n)}^4.
    \end{align*}
\end{proof}

\bibliographystyle{amsplain}
\bibliography{Dyadic_Sobolev_Refs}
\end{document}